\documentclass[12pt,leqno]{amsart}
\usepackage{latexsym,amsmath,amssymb,amscd,xcolor,mathrsfs,marvosym,mathtools}
\usepackage{graphicx,hyperref}
\usepackage{tikz}

\title{An approximate inner ball condition characterizing $C^{1,\alpha}$ convex bodies}
\author{Anthony Cappello}

\address{A.\ Cappello: Department of Mathematics, Niagara University, Niagara University, NY 14109, USA, {\tt cappello@niagara.edu}}

\newcommand{\overbar}[1]{\mkern 1.7mu\overline{\mkern-1.7mu#1\mkern-1.5mu}\mkern 1.5mu}

plus 6pt minus 12pt
\newtheorem{theorem}{Theorem}
\newtheorem{lemma}[theorem]{Lemma}
\newtheorem{corollary}[theorem]{Corollary}
\newtheorem{proposition}[theorem]{Proposition}

\theoremstyle{definition}

\newcommand{\barint}{
\rule[.036in]{.12in}{.009in}\kern-.16in \displaystyle\int }

\newcommand{\barcal}{\mbox{$ \rule[.036in]{.11in}{.007in}\kern-.128in\int $}}

\newcommand{\bbbn}{\mathbb N}

\newcommand{\bbbr}{\mathbb R}
\newcommand{\bbbs}{\mathbb S}
\newcommand{\eps}{\varepsilon}

\def\5{\text{\Saturn}}

\def\dist{\operatorname{dist}}

\def\conv{\operatorname{conv}}

\def\graph{\operatorname{graph}}

\def\epi{\operatorname{epi}}
\def\inte{\operatorname{int}}
\def\co{\operatorname{co}}
\def\AIB{\operatorname{AIB}}

\def\mvint_#1{\mathchoice
          {\mathop{\vrule width 6pt height 3 pt depth -2.5pt
                  \kern -8pt \intop}\nolimits_{\kern -3pt #1}}%
          {\mathop{\vrule width 5pt height 3 pt depth -2.6pt
                  \kern -6pt \intop}\nolimits_{#1}}%
          {\mathop{\vrule width 5pt height 3 pt depth -2.6pt
                  \kern -6pt \intop}\nolimits_{#1}}%
          {\mathop{\vrule width 5pt height 3 pt depth -2.6pt
                  \kern -6pt \intop}\nolimits_{#1}}}

\numberwithin{theorem}{section} \numberwithin{equation}{section}

\begin{document}

\sloppy
\begin{abstract}
    The main result of this paper uses an approximate inner ball condition to characterize $C^{1,\alpha}$ convex bodies for $0<\alpha\leq 1$. In particular, we show that a convex body $K$ is of class $C^{1,\alpha}$  if and only if there exist constants $C,\eps_0>0$ such that for every $0<\eps\leq\eps_0$ and for every $x\in\partial K$ there exists $h^\eps(x)$ satisfying $B(h^\eps(x), C\eps^\frac{1-\alpha}{1+\alpha})\subset K$ and $\dist(x,B(h^\eps(x), C\eps^\frac{1-\alpha}{1+\alpha}))\leq\eps.$ We then give a new proof of the regularity of the convex envelope by showing that $\conv(f)\in C^{1,\alpha}_{\rm loc}(\bbbr^n)$ whenever $f\in C^{1,\alpha}_{\rm loc}(\bbbr^n)$ and $f$ satisfies the coercivity condition $f(x)\to\infty$ as $|x|\to\infty$. Finally, we use the approximate inner ball characterization of $C^{1,\alpha}$ convex bodies to prove that the Minkowski sum of a $C^{1,\alpha}$ convex body and a nonempty closed convex set is a $C^{1,\alpha}$ convex domain. 
\end{abstract}

\maketitle


\section{Introduction}
The uniform inner ball condition is a classical geometric characterization of $C^{1,1}$ regularity for convex bodies. In particular, a convex body $K$ has $C^{1,1}$ boundary if and only if for some $r>0$, the set $K$ is the union of closed balls of radius $r$ contained in $K$ as stated in \cite[Proposition 2.4.3]{hormander}; see also \cite{ACH, cannarsa}. This condition is closely related to the theory of positive reach, applied to the boundary of convex bodies, initiated by Federer \cite{federer} and the regularity of the distance function \cite{krantz-parks-distance}. Related formulations involving Lipschitz normals and two-sided rolling ball conditions for general domains appear in \cite{lewicka, lucas, walther}. 

If one wishes to identify a similar characterization for convex bodies whose boundary has $C^{1,\alpha}$ regularity, with $0<\alpha <1$, then the ability to find a uniform inner ball condition is not possible in general. For example, the epigraph of the function $x_n = |x'|^{1+\alpha}$ has no tangent interior ball at the origin. Yet, it is possible to characterize $C^{1,\alpha}$ convex bodies by use of an approximate inner ball condition, where for each point on the boundary of $K$ we can identify closed balls contained in $K$ at a distance of at most $\eps>0$ whose radius is some fixed constant times $\eps^{(1-\alpha)/(1+\alpha)}$.

For a closed convex set $K\subset \bbbr^n$ and $r>0$ we define
\begin{equation}
\label{eq4}
K_{r}:=\{z\in K: \overbar B(z,r)\subset K\}
\quad
\text{and}
\quad
K(r):= K_r + \overbar B(0,r).
\end{equation}
The set $K_r$ is the inner parallel body of $K$ and $K(r)$ is the union of closed balls of radius $r$ contained in $K$. We say that $K$ satisfies $\AIB(r,\eps)$ if $K_r\neq \emptyset$ and 
\begin{equation}
\label{eq1}
\sup_{x\in K} \dist(x,K(r))\leq \eps.
\end{equation}
Equivalently, every $x\in K$ lies within $\eps$ of at least one closed ball of radius $r>0$ contained in $K$, providing motivation for the notation $\AIB$, approximate inner ball. 

The main result of this paper is the following.
\begin{theorem}
\label{thm1}
Let $K\subset \bbbr^n$ be a convex body, let $0<\alpha\leq 1$ and set 
$$
\beta:=\frac{1-\alpha}{1+\alpha}.
$$
The following conditions are equivalent:
    \begin{itemize}
        \item [(i)] $\partial K$ is a $C^{1,\alpha}$ hypersurface.
        \item [(ii)] There exist $C, \eps_0 >0$ such that $K$ satisfies $\AIB(C\eps^\beta, \eps)$ for every $0<\eps\leq \eps_0$.
    \end{itemize}
Moreover, the implication $(ii)\implies (i)$ remains valid if a convex body is replaced by a closed convex domain.
\end{theorem}

We note that for $\alpha = 1$, the radius $C\eps^\beta = C$ is independent of $\eps$, so that letting $\eps\searrow 0$ recovers the usual uniform inner ball characterization. 

The proof has two short geometric arguments. In the forward direction, after a change of coordinates, the boundary of $K$ can be represented locally as the graph of a function which lies beneath a $(1+\alpha)$-paraboloid, that is, a function of the form $f(x) = a|x|^{1+\alpha}+b\cdot x + c$. We then show that a radius $C\eps^\beta$ ball is contained in a local epigraph chart of the $(1+\alpha)$-paraboloid, and by restricting further if necessary, contained in the original chart, hence in the interior of $K$, at a distance of at most $\eps$ from the boundary. In the reverse direction, fix $x\in\partial K$ and $R>0$ such that $K_R\neq \emptyset$. We define the unit vector $v_x = (\pi_{K_{R}}(x) - x)/|\pi_{K_{R}}(x) - x|$, which depends on the choice of $R$. Then in Lemma \ref{lem9} we show that $v_x$ is $1/R$-Lipschitz as a function of $x$, that is,
\begin{equation}
\label{eq3}
|v_x - v_y|\leq \frac{|x-y|}{R}.
\end{equation}
Consider any inward supporting unit normal $\nu$ at $x\in\partial K$. We show in Lemma \ref{lem6} that if $K$ satisfies $\AIB(R,\eps)$, then
\begin{equation}
\label{eq2}
|\nu - v_x|^2 \leq \frac{2\eps}{R}.
\end{equation}
Taking $R= C\eps^\beta$ and $\eps = |x-y|^{1+\alpha}$ in \eqref{eq3} and \eqref{eq2} gives the $\alpha$-H\"{o}lder estimate for the normal.

To apply this result to convex functions, we define a similar approximate inner ball characterization for the epigraphs of functions. For a function $f:U\to\bbbr$ and $V\subset U$, letting $\Gamma_f(x) = (x,f(x))$, we say that $\epi(f)$ satisfies $\AIB_V(R,\eps)$ if for every $y\in V$ there is a closed ball $B_y\subset \epi(f)$ of radius $R$ such that $\dist(\Gamma_f(y), B_y)\leq \eps$. Using this definition, we can establish an epigraph form of Theorem \ref{thm1}.
\begin{corollary}
\label{cor1}
Let $f:\bbbr^n\to\bbbr$, $0<\alpha\leq 1$ and $\beta = (1-\alpha)/(1+\alpha)$.
    \begin{itemize}
        \item [(a)] If $f\in C^{1,\alpha}_{\rm loc}(\bbbr^n)$, then for every $x\in\bbbr^n$ and $r>0$ there exist $C, \eps_0>0$ such that $\epi(f)$ satisfies $\AIB_{\overbar B(x,r)}(C\eps^\beta, \eps)$ for every $0<\eps\leq \eps_0$.
        \item [(b)] If $f:\bbbr^n\to\bbbr$ is convex and for every $x\in\bbbr^n$ and $r>0$ there exist $C, \eps_0>0$ such that $\epi(f)$ satisfies $\AIB_{\overbar B(x,r)}(C\eps^\beta, \eps)$ for every $0<\eps\leq \eps_0$, then $f\in C^{1,\alpha}_{\rm loc}(\bbbr^n)$.
    \end{itemize}
\end{corollary}

Both Theorem \ref{thm1} and Corollary \ref{cor1} are useful for understanding regularity of convex bodies and functions because balls behave well under convex combinations and translations. We use Corollary \ref{cor1} to recover the Kirchheim-Kristensen regularity theorem for convex envelopes \cite{kirc and kris}; related regularity results can be found in \cite{figalli, rabier, kiselman}. A general modulus version is presented in \cite[Theorem 2.3]{fathi} and used there to prove a $C^{1,1}$ insertion theorem. For PDE
characterizations, regularity, and applications of convex envelopes, see \cite{oberman, oberman-silvestre, pimentel-shahgholian}.

\begin{theorem}
\label{thm2}
Let $f\in C^{1,\alpha}_{\rm loc}(\bbbr^n)$ for $0<\alpha\leq 1$ and $f(x)\to\infty$ as $|x|\to\infty$. Then $\conv(f)\in C^{1,\alpha}_{\rm loc}(\bbbr^n)$.
\end{theorem}
The new proof of Theorem \ref{thm2} presents a geometrically motivated argument. 
We also obtain an elementary and short proof of the preservation of $C^{1,\alpha}$ regularity under Minkowski addition of a convex body with boundary of class $C^{1,\alpha}$ and a nonempty closed convex set, a result established by Krantz and Parks \cite[Theorem 3]{KrantzP}; see also \cite{kiselman2}.
\begin{theorem}
\label{thm3}
If $A\subset\bbbr^n$ is a convex body with boundary of class $C^{1,\alpha}$, $0<\alpha\leq 1$, and $B\subset\bbbr^n$ is any nonempty closed convex set, then $A+B$ is a convex domain with boundary of class $C^{1,\alpha}$.
\end{theorem}
The approximate inner ball characterization given in Theorem \ref{thm1} allows us to avoid the need to write the boundaries in coordinates, providing a simple one-paragraph proof for this result.

The paper is organized as follows. Section 2 contains the notation and preliminary material necessary to read and understand the results and their proofs. Sections 3 and 4 prove the two implications in Theorem \ref{thm1}. Section 5 supplies the proof of Corollary \ref{cor1}. Section 6 provides the details for the regularity of the convex envelope and Section 7 provides the details for the preservation of regularity under Minkowski addition.


\section{Preliminaries}
We denote the $n$-dimensional open (closed) ball centered at $x\in\bbbr^n$ with radius $r>0$ as $B^n(x,r)$ ($\overbar B^n(x,r)$) and when context is clear we denote the open (closed) ball as $B(x,r)$ ($\overbar B(x,r)$). Let $a,r>0$ and $x, h\in\bbbr^n$. It is easy to verify that
$$
aB(h,r)+\{x\} = \{ay+x: y\in B(h,r)\} = B(ah + x, ar).
$$
The graph of $f:U\to\bbbr$ is denoted and defined by $\graph(f):=\{(x,f(x)):x\in U\}$ and we denote specific points on the graph as $\Gamma_f(x) = (x,f(x))$.
We define the distance between the point $x\in\bbbr^n$ and the set $K\subset\bbbr^n$ by:
$$
\dist(x,K) = \inf_{y\in K} |x-y|
$$
and we note that if $K$ is nonempty and closed, then this distance is achieved at some $y\in K$. The line segment from $x\in\bbbr^n$ to $y\in\bbbr^n$ is denoted as $[x,y]$.

We say a set $A$ is convex if for all $x,y\in A $, we have $[x,y]\subset A$. Moreover we denote the convex hull of $A\subset\bbbr^n$ as $\co(A)$ and define it as
$$
\co(A) := \left\{\sum_{i=1}^{m} \lambda_ix_i : \sum_{i=1}^{m} \lambda_i = 1,\, \lambda_i \geq 0, \,x_i\in A,\,m\in\bbbn\right\}.
$$
The convex hull can be thought of as the set of convex combinations of points of $A$, but in fact at most $n+1$ points of $A$ are required as demonstrated in Carath\'{e}odory's Theorem, see \cite[Chapter A, Theorem 1.3.6]{hiriart}. 

\begin{theorem}[Carath\'{e}odory's Theorem]
\label{thm 14}
For any subset $A\subset \bbbr^n$, its convex hull admits the representation
$$
\co(A) = \left\{\sum_{i=1}^{n+1} \lambda_ix_i : \sum_{i=1}^{n+1} \lambda_i = 1,\, \lambda_i \geq 0, \,x_i\in A\right\}.
$$
\end{theorem}

If $C\subset\bbbr^n$ is a nonempty closed convex set, then for every $x\in\bbbr^n$, there is a unique point, called the metric projection of $x$ onto $C$, denoted by $\pi_C(x)$, such that
\begin{equation}
\label{eq 17}
\pi_C(x)\in C
\qquad
\text{and}
\qquad
|x-\pi_C(x)|=\dist(x,C).
\end{equation}
We then have the following simple property of the metric projection.
\begin{proposition}
\label{prop2}
Let $C$ be nonempty, closed, and convex. Then for all $x\in\bbbr^n$ and $z\in C$,
$$
\langle x-\pi_C(x), z- \pi_C(x)\rangle \leq 0.
$$
\end{proposition}

\begin{proof}
For any $z\in C$, by convexity, $tz + (1-t)\pi_C(x)\in C$ for $t\in(0,1)$ and hence by \eqref{eq 17}
$$
0 \leq |x - (tz + (1-t)\pi_C(x))|^2 - |x-\pi_C(x)|^2 = t^2|z-\pi_C(x)|^2 - 2t \langle x-\pi_C(x), z- \pi_C(x)\rangle.
$$
Dividing by $t>0$ and letting $t\to 0^+$ completes the proof.
\end{proof}

If $C$ is nonempty, closed, and convex, then for any $x\notin C$, the unit vector $\frac{x-\pi_C(x)}{|x-\pi_C(x)|}$ acts as an outer supporting normal to $C$ at $\pi_C(x)$. The next lemma establishes Lipschitz properties for numerous maps involving the metric projection.
\begin{lemma}
\label{lem9}
Let $C\subset\bbbr^n$ be nonempty, closed, and convex. Then both $\pi_C$ and $\operatorname{Id} - \pi_C$ are 1-Lipschitz. Moreover if $x,y\in\bbbr^n$ and $R>0$ are such that $\dist(x,C), \dist(y,C) \geq R$, then
\begin{equation}
\label{eq16}
\left|\frac{x-\pi_C(x)}{|x-\pi_C(x)|} - \frac{y-\pi_C(y)}{|y-\pi_C(y)|}\right| \leq \frac{1}{R}|x-y|.
\end{equation}
\end{lemma}

\begin{proof}
To prove $\pi_C$ is 1-Lipschitz, apply Proposition \ref{prop2} using $z=\pi_C(y)$, switch the roles of $x$ and $y$ and apply Proposition \ref{prop2} again to get the inequalities
$$
\langle x-\pi_C(x), \pi_C(y) - \pi_C(x)\rangle\leq 0,
\quad
\langle y-\pi_C(y), \pi_C(x) - \pi_C(y)\rangle\leq 0.
$$
Combining these two inequalities yields 
\begin{equation}
\label{eq 61}
\langle x-\pi_C(x) - (y-\pi_C(y)), \pi_C(y) - \pi_C(x)\rangle \leq 0
\end{equation}
and we simplify to get
$$
|\pi_C(y)-\pi_C(x)|^2\leq \langle x-y, \pi_C(x)-\pi_C(y)\rangle.
$$
Then the Cauchy-Schwarz inequality shows $\pi_C$ is 1-Lipschitz. Then applying \eqref{eq 61} we have
\begin{align}
\langle x-\pi_C(x) - (&y-\pi_C(y)), x-y\rangle \notag \\
&= |x-\pi_C(x) - (y-\pi_C(y))|^2 + \langle x-\pi_C(x) - (y-\pi_C(y)), \pi_C(x) - \pi_C(y)\rangle \notag \\ 
&\geq |x-\pi_C(x) - (y-\pi_C(y))|^2 \notag
\end{align}
and applying Cauchy-Schwarz shows that $\operatorname{Id}-\pi_C$ is 1-Lipschitz.

To prove \eqref{eq16} let $a = |x-\pi_C(x)|$, $b = |y-\pi_C(y)|$, and $c = \left\langle \frac{x-\pi_C(x)}{|x-\pi_C(x)|}, \frac{y-\pi_C(y)}{|y-\pi_C(y)|} \right\rangle$. Then,
\begin{equation}
\label{eq 62}
\left|\frac{x-\pi_C(x)}{a} - \frac{y-\pi_C(y)}{b}\right|^2 = 2(1-c)
\end{equation}
and we also have
\begin{equation}
\label{eq14}
|x-\pi_C(x) - (y-\pi_C(y))|^2 = a^2 + b^2 - 2abc = (a - b)^2 +2ab(1-c)\geq 2ab(1-c).
\end{equation}
Given that $a, b \geq R$, \eqref{eq14} and \eqref{eq 62} imply
$$
|x-\pi_C(x) - (y-\pi_C(y))|^2 \geq 2R^2(1-c) = R^2 \left|\frac{x-\pi_C(x)}{a} - \frac{y-\pi_C(y)}{b}\right|^2.
$$
Finally, as $\operatorname{Id}-\pi_C$ is 1-Lipschitz, this proves \eqref{eq16}.
\end{proof}

For $U\subset \bbbr^n$ convex, we say a function $f:U\to\bbbr$ is convex if for every $x,y\in U$ and $t\in [0,1]$,
$$
f(tx+(1-t)y)\leq tf(x) + (1-t)f(y).
$$
Thus we can connect convex functions with convex sets by use of the epigraph of a function $f:U\to\bbbr$, which is defined as
$$
\epi(f) = \{(x,t)\in U\times \bbbr:t\geq f(x)\}.
$$
It is a standard result, and easy to prove, that if $U$ is convex, then $f$ is convex if and only if $\epi(f)$ is convex \cite[Theorem 1.37]{morduk}, and hence many results applying to convex sets can be applied to convex functions.

We say $K$ is a convex body if $K$ is convex, compact, and has non-empty interior. Similarly we define a convex domain as a convex set with nonempty interior. Given a closed convex domain $K$, for each $x\in \partial K$, there exists an inner supporting unit normal $\nu\in\bbbs^{n-1}$ such that
\begin{equation}
\label{eq13}
\langle y-x, \nu \rangle \geq 0
\quad
\text{ for all }
y\in K
\end{equation}
and we define a supporting hyperplane of $K$ at $x$ as a hyperplane through $x$ with inward supporting unit normal $\nu$. Thus, for $U\subset\bbbr^n$ convex, we can define a supporting hyperplane of a convex function $f:U\to\bbbr$ at $x\in U$ to be a supporting hyperplane of $\epi(f)$ at $\Gamma_f(x)$. 

We recall the sets $K_r$ and $K(r)$, which are defined in \eqref{eq4}. The next proposition then clarifies that these sets are also convex bodies for sufficiently small $r$.
\begin{proposition}
\label{prop 9}
Let $K\subset\bbbr^n$ be a convex body such that $K$ contains a ball of radius $r_0>0$. Then $K_r$ and $K(r)$ are both convex bodies for all $r\in (0,r_0)$.
\end{proposition}

\begin{proof}
Let $x,y\in K_r$. We show that $[x,y]\subset K_r$. Clearly, 
$\overbar{B}(x,r),\overbar{B}(y,r)\subset K$ and for any $z\in [x,y]$,
$
\overbar{B}(z,r)\subset \operatorname{co}(\overbar{B}(x,r)\cup\overbar{B}(y,r))\subset K.
$
Thus, $\dist(z,\partial K)\geq r$ and $z\in K_r$ imply $[x,y]\subset K_r$ showing $K_r$ is convex. Moreover, if $B(x,r_0)\subset K$, then for any $r\in (0,r_0)$ the ball $B(x,r_0-r)\subset K_r$ proving that the interior of $K_r$ is non-empty. Finally, as $K_r = \bigcap_{z\in \overbar B(0,r)} K-\{z\}$, we have that $K_r$ is compact. We conclude the proof by noting that $K(r)$ is convex and compact being the sum of two convex compact sets and as $K(r)$ contains a ball it has nonempty interior.
\end{proof}

For $U\subset\bbbr^n$ open, a function $f:U\to\bbbr$ is $C^{1,\alpha}(U)$ for $\alpha\in (0,1]$ if $f\in C^1(U)$ and $\nabla f$ is $\alpha$-H\"{o}lder continuous (Lipschitz continuous when $\alpha =1$) on $U$, that is, there exists $L>0$ such that for all $x,y\in U$,
$$
|\nabla f(x)- \nabla f(y)| \leq L|x-y|^\alpha.
$$
We call $L$ a H\"{o}lder constant of $\nabla f$. Similarly, we say $f\in C^{1,\alpha}_{\rm loc}(U)$ if for every $x\in U$ and $\delta>0$ such that $\overbar B(x,\delta)\subset U$, we have $f\in C^{1,\alpha}(B(x,\delta))$. 

For $x\in\bbbr^n$, $v\in\bbbs^{n-1}$, and $R, M>0$, we define the open bounded cylinder
$$
C(x, v, R, M):=\{x+u+tv: u\in \{v\}^\perp,\,|u|<R, \text{ and } |t|<M\}
$$
and when $R = M$ we simply write $C(x, v, M)$. We also denote the point $x = (x',x_n)\in\bbbr^n$, where $x'\in\bbbr^{n-1}$. Using this notation we then say for a domain $U\subset \bbbr^n$, the boundary of $U$ is a hypersurface of $\bbbr^n$, if for each $x\in \partial U$, there exists $v\in\bbbs^{n-1}$ and constants $R, M>0$ such that after a rotation and translation we have a function $f_x:B^{n-1}(0,R)\to \bbbr$ satisfying
$$
C(0, e_n, R, M) \cap \partial U = \{(y',f_x(y')):y'\in B^{n-1}(0',R)\}
$$
$$
C(0, e_n, R, M) \cap  \overbar U = \{(y',t): y'\in B^{n-1}(0',R),\, f_x(y')\leq t < M  \}.
$$
The function $f_x$ combined with the sets $C(0, e_n, R, M) \cap \partial U$ and $C(0, e_n, R, M) \cap U$ are called the charts of $\partial U$. Simply put, the boundary of a domain is a hypersurface when its boundary can be represented locally as the graph of a function. If $f_x$ is convex for each $x\in \partial U$, then we say $\partial U$ is a convex hypersurface. If $f_x\in C^1$ for each $x\in \partial U$, then $\partial U$ is a $C^1$ hypersurface. If $f_x\in C^{1,\alpha}$ for $0<\alpha \leq 1$ and for all $x\in\partial U$, then $\partial U$ is a $C^{1,\alpha}$ hypersurface. For convex bodies (or closed convex domains) we say the boundary is of class $C^{1,\alpha}$ if the boundary is a $C^{1,\alpha}$ hypersurface. 

The next proposition states that the boundary of a closed convex domain is a convex hypersurface, that is, it can be locally represented as the graph of a convex function. 
\begin{proposition}
\label{prop5}
For a closed convex domain $K\subset \bbbr^n$, and for any $x\in \partial K$, there exists $v\in\bbbs^{n-1}$ and $M, R>0$ such that $\partial K\cap C(x, v, R, M)$ is the graph of a convex function defined on an open ball in a supporting hyperplane of $K$ at $x$ whose inward unit normal is $v$.
\end{proposition}
\begin{proof}
For each $x\in\partial K$ we rotate and translate the convex domain so that $x=0$, $te_n\in\inte K$ for some value of $t>0$ and $K$ is contained in the upper half space $\{x\in\bbbr^n:x_n\geq 0\}$, implying $e_n$ is an inner supporting unit normal. Then we can find $R, M>0$ so that $B^{n-1}(0,R)\times \{M\}\subset \inte K$ and thus the cylinder $C(0,e_n, R, M)$ intersects $\partial K$ in such a way that the projection of $C(0,e_n, R, M)\cap \partial K$ onto the plane $x_n =0$ is $B^{n-1}(0,R)$. Then we note that the function $f:B^{n-1}(0,R)\to\bbbr$ defined as
$$
f(x) = \inf \{x_n :(x, x_n)\in K\}
$$
is finite, being bounded below by the plane $x_n=0$, and achieves its infimum as $K$ is closed. Moreover, as $K$ is convex so too is $f$. 
\end{proof}
It is well known that a convex function $f$ is differentiable at $x\in\inte U$ if and only if $f$ has a unique supporting hyperplane at $x$ \cite[Theorem 25.1]{rockafellar}, and if $f$ is differentiable and convex on the open convex set $U\subset\bbbr^n$, then $f$ is $C^1(U)$ \cite[IV Theorem E]{roberts}. Therefore we have the following characterization of $C^1$ convex domains.
\begin{proposition}
\label{prop 2}
A convex domain $K\subset \bbbr^n$ has a unique supporting hyperplane at each $x\in \partial K$ if and only if $\partial K$ is a $C^1$ hypersurface.
\end{proposition}

It is equivalent to say that a closed convex domain $K$ has boundary of class $C^{1,\alpha}$ if and only if $K$ has $C^1$ boundary and the inner unit normal vector of $K$ is locally $\alpha$-H\"{o}lder continuous. 
\begin{proposition}
\label{prop 5}
Let $K\subset \bbbr^n$ be a closed convex domain and let $0<\alpha\leq 1$. The boundary $\partial K$ is of class $C^{1,\alpha}$ if and only if for every $x\in\partial K$, there exists a unique inward unit supporting normal $\nu(x)$, where $\nu(x)$ is locally $\alpha$-H\"{o}lder on $\partial K$.
\end{proposition}

\begin{proof}
We define the map $N:\bbbr^{n-1}\to\bbbr^{n}$ by
$$
N(u):= \frac{(-u,1)}{\sqrt{1+|u|^2}}.
$$
Then we have $\Vert DN\Vert \leq 1$, so that $N$ is a 1-Lipschitz map. Therefore for any $U$ open and $f\in C^{1,\alpha}(U)$, with a H\"{o}lder constant of $L>0$, the inward normal of the graph of $f$ at $\Gamma_f(x)$ can be represented as $\nu(\Gamma_f(x)) = N(\nabla f(x))$ showing that for all $x,y\in U$,
$$
|\nu(\Gamma_f(x)) - \nu(\Gamma_f(y))|\leq |\nabla f(x) -\nabla f(y)|\leq L|x-y|^\alpha \leq L|\Gamma_f(x) - \Gamma_f(y)|^\alpha
$$
showing that the normal $\nu$ is locally H\"{o}lder on the graph of $f$. The forward direction follows by writing the boundary in charts and applying this argument. For the converse statement, the uniqueness of $\nu(x)$ shows that $\partial K$ is $C^1$ by Proposition \ref{prop 2}. Fix $x\in\partial K$, and by rotating and translating so that $x=0$ and $e_n$ is an inward supporting unit normal, we apply Proposition \ref{prop5} providing the existence of $M, R>0$ such that $C(0, e_n, 2R, M)\cap \partial K$ is the graph of $f\in C^1(B^{n-1}(0',2R))$. Also, by decreasing $R>0$ if necessary, we can assume that $\nu$ is $\alpha$-H\"{o}lder continuous on $C(0, e_n, 2R, M)\cap \partial K$. Restricting $f$ to $B^{n-1}(0',R)$, there exists $A>0$ such that $|\nabla f(x')|<A$ for all $x'\in B^{n-1}(0',R)$. Then, for all $x',y'\in B^{n-1}(0',R) $
\begin{equation}
\label{eq 53}
|\Gamma_f(x') - \Gamma_f(y')|^2 = |x' - y'|^2 + |f(x') - f(y')|^2 \leq (1+A^2)|x'-y'|^2
\end{equation}
showing that $\Gamma_f$ is $\sqrt{1+A^2}$-Lipschitz on $B^{n-1}(0',R)$. Moreover, if we denote $\nu=(\nu', \nu_n)$, then
\begin{equation}
\label{eq 63}
\nu_n(\Gamma_f(x')) = N_n(\nabla f(x')) = \frac{1}{\sqrt{1+|\nabla f(x')|^2}} \geq \frac{1}{\sqrt{1+A^2}}>0.
\end{equation}
Then we see that the map $\phi(z) = -z'/z_n$ has bounded derivative, and hence is Lipschitz, on $\{z\in\bbbr^n: |z|\leq 1,\, z_n \geq (1+A^2)^{-1/2}\}$. Therefore, given that $\phi$ and $\Gamma_f$ are Lipschitz and $\nu$ is $\alpha$-H\"{o}lder continuous, and noting
$$
\nabla f(x') = -\frac{\nu'(\Gamma_f(x'))}{\nu_n(\Gamma_f(x'))} = \phi(\nu(\Gamma_f(x')))
$$
we then have that $f\in C^{1,\alpha}(B^{n-1}(0',R))$ as desired.
\end{proof}

\section{The Forward Implication of Theorem \ref{thm1}}
\label{S3}
To show that a $C^{1,\alpha}$ convex body satisfies the desired approximate inner ball condition, we need only show that the boundary satisfies the approximate inner ball condition.
\begin{proposition}
\label{prop3}
Let $K$ be a convex body and suppose that $R>0$ satisfies $K_{R}\neq \emptyset$. If $\sup_{x\in\partial K}\dist(x, K(R))\leq \eps$, then $K$ satisfies $\AIB(R,\eps)$. 
\end{proposition}

\begin{proof}
Since $\partial K\subset K(R) +\overbar B(0,\eps)$ and $\co(\partial K) = K$, by convexity, we have $K\subset K(R) + \overbar B(0,\eps)$ so that 
$
\sup_{x\in K} \dist (x, K(R))\leq \eps.
$
\end{proof}

The following Taylor estimate gives, at each point of the graph of $f$, a touching $(1+\alpha)$-paraboloid lying in $\epi(f)$.
\begin{lemma}
\label{lem1}
Let $f\in C^{1,\alpha}(B(z,r))$ with $\sup\left\{\frac{|\nabla f(x) - \nabla f(y)|}{|x-y|^\alpha}: x,y \in B(z,r),\, x\neq y\right\}\leq L$. Then for all $x,y\in B(z,r)$, 
$$
f(x)\leq \frac{L}{1+\alpha}|x-y|^{1+\alpha} +\nabla f(y)\cdot (x-y) + f(y).
$$
\end{lemma}

\begin{proof}
The Fundamental Theorem of Calculus implies
$$
f(x) - f(y) = \int_0^1 \nabla f((1-t)y + tx)\cdot (x-y)\,dt
$$
and using the H\"{o}lder continuity of the derivative we have for all $x,y\in B(z,r)$,
\begin{align}
    f(x) - f(y) - \nabla f(y)\cdot(x-y)
    &= \int_0^1 (\nabla f((1-t)y + tx) - \nabla f(y))\cdot(x-y)\,dt \notag \\
    &\leq \int_0^1 |\nabla f((1-t)y + tx) - \nabla f(y)||x-y|\,dt \notag \\
    &\leq L|x-y| \int_0^1 |(1-t)y+tx - y|^\alpha\,dt \notag \\
    &= \frac{L}{1+\alpha}|x-y|^{1+\alpha}. \notag
\end{align}

\end{proof}
While we are concerned with the approximate inner ball condition in this paper, we need the uniform inner ball condition that the paraboloid satisfies. 
\begin{lemma}
\label{lem 16}
Let $a > 0$, $b\in\bbbr^n$, $c\in\bbbr$, and define the paraboloid $p:\bbbr^n\to\bbbr$ by $p(x) = a|x|^2 + b\cdot x + c$. For every $\Gamma_p(x)$, there exists $h(x)\in \epi(p)$ such that $\Gamma_p(x)\in\overbar B\left(h(x), \frac{1}{2a}\right)\subset \epi(p)$.
\end{lemma}

\begin{proof}
We first assume that $p$ is of the form $p(x) = a|x|^2$. By examining the behavior of the normal line at every point on the graph of $p$ we can see that the set of points which do not have a unique metric projection onto the graph of $p$, defined as the medial axis, is equal to the set $\left\{x\in\bbbr^{n+1}: x_1=\cdots = x_{n} = 0 \text{ and } x_{n+1} > \frac{1}{2a}\right\}$. Moreover, every point on the medial axis is at distance at least $1/(2a)$ from $\graph(p)$. The result follows by noting that $p(x) = a|x|^2 + b\cdot x + c$ is a translated paraboloid of the form $a|x|^2$.
\end{proof}
The main tool of the proof of the forward implication for Theorem \ref{thm1} comes from the next lemma.

\begin{lemma}
\label{lem 17}
For $0<\alpha\leq 1$, let $f\in C^{1,\alpha}(B^{n-1}(0,2r))$ for some $r>0$ and let $\beta = (1-\alpha)/(1+\alpha)$. Then there exist constants $C,\eps_0>0$ such that for all $0<\eps\leq \eps_0$, the  set $\epi(f)$ satisfies $\AIB_{B^{n-1}(0,r)}(C\eps^\beta,\eps)$.
\end{lemma}

\begin{proof}
As $f\in C^{1,\alpha}(B^{n-1}(0,2r))$, let $L>0$ be such that for all $x,y\in B^{n-1}(0,2r)$ we have
$$
|\nabla f(x) - \nabla f(y)|\leq L|x-y|^\alpha.
$$
By Lemma \ref{lem1} we know that for all $x,y\in B^{n-1}(0,2r)$,
\begin{equation}\label{eq 7}
   f(x)\leq \frac{L}{1+\alpha}|x-y|^{1+\alpha} +\nabla f(y)\cdot (x-y) + f(y) .
\end{equation}
Hence, for fixed $y\in B^{n-1}(0,2r)$ and any $\eps>0$, we can define two functions, $g_y:\bbbr^{n-1}\to\bbbr$ and the paraboloid $p^\eps_y:\bbbr^{n-1}\to \bbbr$ by
$$
g_y(x) = \frac{L}{1+\alpha}|x-y|^{1+\alpha} + \nabla f(y)\cdot (x-y) + f(y)
$$
$$
p_y^\eps(x) = \frac{L^{1+\beta}}{(1+\alpha)\eps^\beta}|x-y|^2 + \nabla f(y)\cdot(x-y) + f(y).
$$
We note that $f(y) = g_y(y) = p_y^\eps(y)$ and by \eqref{eq 7} for all $x\in B^{n-1}(0,2r)$, we have  $f(x) \leq g_y(x)$. We now claim that for all $x\in \bbbr^{n-1}$, $y\in B^{n-1}(0,2r)$, and $\eps>0$, we have $g_y(x) \leq p_y^\eps(x) + \eps$, which is equivalent to showing for all $t>0$
$$
\frac{L}{1+\alpha}t^{1+\alpha} \leq \frac{L^{1+\beta}}{(1+\alpha)\eps^\beta} t^2+\eps.
$$
The inequality is obvious when $\alpha =1$ so we assume $0<\alpha<1$. If $t\leq(\eps/L)^{1/(1+\alpha)}$, then $Lt^{1+\alpha}\leq \eps$ implying
$$
\frac{L}{1+\alpha}t^{1+\alpha} \leq \frac{\eps}{1+\alpha}\leq \frac{L^{1+\beta}}{(1+\alpha)\eps^\beta} t^2 + \eps.
$$
If $t\geq(\eps/L)^{1/(1+\alpha)}$, then $t^{\alpha-1}\leq (L/\eps)^\beta$ so that
$$
\frac{L}{1+\alpha}t^{1+\alpha} = \frac{L}{1+\alpha}t^{\alpha-1} t^2 \leq \frac{L}{1+\alpha}\left(\frac{L}{\eps}\right)^\beta t^2 = \frac{L^{1+\beta}}{(1+\alpha)\eps^\beta} t^2 \leq \frac{L^{1+\beta}}{(1+\alpha)\eps^\beta} t^2 + \eps.
$$
In either case, we have proven the desired inequality and for all $\eps>0$ and for all $x,y\in B^{n-1}(0,2r)$,
\begin{equation}
\label{eq 21}
f(x) \leq g_y(x) \leq p^\eps_y(x)+\eps.
\end{equation}

We note that $p_y^\eps +\eps$ is a paraboloid and hence by Lemma \ref{lem 16}, for every $y\in B^{n-1}(0,2r)$ and $x\in\bbbr^{n-1}$ and for any $0<C\leq (1+\alpha)/(2L^{1+\beta})$, there exists $h_y^\eps(x)\in \epi(p_y^\eps + \eps)$ such that 
\begin{equation}
\label{eq 40}
(x, p_y^\eps(x) + \eps)\in \overbar B^n(h_y^\eps(x), C\eps^\beta) \subset \epi(p_y^\eps + \eps).
\end{equation}
For any $y\in B^{n-1}(0,2r)$, as $f(y) = p_y^\eps(y)$, we then have by \eqref{eq 40},
\begin{equation}
\label{eq 41}
\dist\big(\Gamma_f(y),\overbar B^n (h^\eps(y), C\eps^\beta)\big) \leq \eps,
\end{equation}
where for simplicity we rewrite $h_y^\eps(y)$ as $h^\eps(y)$. It remains to show that there exists some $\eps_0>0$ such that for all $0<\eps\leq \eps_0$ and for all $y\in B^{n-1}(0,r)$, we have $\overbar B^n(h^\eps(y), C\eps^\beta)\subset \epi(f)$. To accomplish this, when $\alpha\in(0,1)$, we fix $\eps_0>0$ so that 
\begin{equation}
\label{eq 20}
\eps_0 + 2C\eps_0^\beta<r.
\end{equation}
If $\alpha = 1$, let $0<\eps_0<r/2$ and choose $C\leq\min\{1/L, r/4\}$ so that $\eps_0+2C<r$.
Let $0<\eps\leq \eps_0$. Then for $y\in B^{n-1}(0,r)$, by \eqref{eq 41} and \eqref{eq 20} we have
\begin{equation}
\label{eq36}
\overbar B^n(h^\eps(y), C\eps^\beta) \subset \overbar B^n(\Gamma_f(y), r) \subset B^{n-1}(0,2r) \times (-\infty, \infty).
\end{equation}
By \eqref{eq 21}, we have $p_y^\eps(x) + \eps \geq f(x)$ for all $x\in B^{n-1}(0,2r)$ so that by \eqref{eq 40} we have
$$
\overbar B^n(h^\eps(y), C\eps^\beta)\subset \big(B^{n-1}(0,2r)\times (-\infty,\infty)\big)\cap \epi(p_y^\eps + \eps) \subset \epi(f)
$$
as desired. Therefore we have shown that there exist constants $C,\eps_0>0$ such that for all $0<\eps\leq \eps_0$ and for each $x\in B^{n-1}(0,r)$ there exists $h^\eps(x)\in \epi(f)$ such that $\overbar B^{n}\big(h^\eps(x), C\eps^\beta\big)\subset \epi(f)$ and
$$
\dist\big(\Gamma_f(x),\overbar B^{n}\big(h^\eps(x), C\eps^\beta\big)\big) \leq \eps
$$
which is precisely the definition of $\epi(f)$ satisfying $\AIB_{B^{n-1}(0,r)}(C\eps^\beta,\eps)$.
\end{proof}

We now turn to the proof of $(i)\implies (ii)$ for Theorem \ref{thm1}. We know by Proposition \ref{prop 5}, for the convex body $K$, with boundary of class $C^{1,\alpha}$, for each $x\in\partial K$ there exist $R, L>0$ such that after translating $x\in\partial K$ so that $x\mapsto 0$ and rotating $\nu(x)$ so that $\nu(x)\mapsto e_n$ we have
\begin{equation}
\label{eq6}
\partial K\cap C(0,e_n, R) = \{(x', x_n) : 0\leq x_n = f(x') < R \text{ and } |x'|< R\}
\end{equation}
where
\begin{equation}
\label{eq7}
f(0) = 0,
\quad
\nabla f(0) = 0,
\quad
|\nabla f(x') - \nabla f(y')| \leq L|x' - y'|^\alpha \text{ for all } |x'|, |y'| < R. 
\end{equation}

\begin{proof}[Proof of Forward Direction of Theorem \ref{thm1}]

Let $K\subset\bbbr^n$ be a convex body with boundary of class $C^{1,\alpha}$. For any $x\in\partial K$ by translating and rotating we can assume, by using \eqref{eq6} and \eqref{eq7}, that $x=0$ and
$$
C(0, e_n, 2r, M) \cap K = \{x : |x'|<2r,\, 0\leq f(x')\leq x_n < M\}.
$$
Moreover, by shrinking $r$ if necessary, we may assume, by continuity,
$$
m:=\sup_{|x'|\leq r} f(x') < M.
$$
From the proof of Lemma \ref{lem 17} we know there exist $C_x, \eps_x>0$ such that for all $0<\eps\leq \eps_x$ and for all $|y'|<r$, there exists $h^\eps(y')$ satisfying $\overbar B^{n}(h^\eps(y'), C_x\eps^\beta) \subset \epi(p_{y'}^\eps + \eps)$ and
\begin{equation}
\label{eq 31}
\sup_{z\in \overbar B^{n}(h^\eps(y'), C_x\eps^\beta)}|z-\Gamma_f(y')|\leq \eps+2C_x\eps^\beta,
\quad
\dist(\Gamma_f(y'), \overbar B^{n}(h^\eps(y'), C_x\eps^\beta))\leq \eps.
\end{equation}
Set $\delta = \min\{r, M-m\}>0$. For $\alpha \in (0,1)$ we can decrease $\eps_x$ so that $\eps_x + 2C_x\eps_x^\beta<\delta$. For $\alpha = 1$, we can reduce $C_x$ so that $2C_x<\delta/2$ and then, if necessary, reduce $\eps_x$ so that $\eps_x<\delta/2$ and $\eps_x+2C_x<\delta$. Then by \eqref{eq 31} we have for all $z\in \overbar B^{n}(h^\eps(y'), C_x\eps^\beta)$,
$$
|z'|<2r
\quad
\text{ and }
\quad
f(z') \leq z_n \leq f(y')+\eps + 2C_x \eps^\beta < M
$$
and hence $z\in K$ showing $\overbar B^n(h^\eps(y'),C_x\eps^\beta)\subset K$.

The proof then follows by a compactness argument. Specifically, there exist $x_1,\ldots,x_N\in\partial K$, such that 
$$
\partial K \subset \bigcup_{i=1}^N C(x_i, \nu(x_i), r_i, M_i)
$$
and for each $x_i$ we know there exists $C_i:=C_{x_i}>0$ and $\eps_i:=\eps_{x_i}>0$ such that for all $y\in \partial K\cap C(x_i, \nu(x_i), r_i, M_i)$, we have
$$
\dist(y, K(C_i\eps^\beta))\leq \eps
\quad
\text{for all }
0<\eps\leq \eps_i.
$$
We then choose $C = \min\{C_1,\ldots, C_N\}$ and $\eps_0 = \min\{\eps_1,\ldots, \eps_N\}$. Therefore for any $y\in\partial K$, there is some $i\in \{1,\ldots,N\}$ such that $y\in \partial K\cap C(x_i,\nu(x_i), r_i, M_i)$ and as $K(C_i\eps^\beta)\subset K(C\eps^\beta)$,
$$
\dist(y, K(C\eps^\beta))\leq \dist(y, K(C_i\eps^\beta))\leq \eps
\quad
\text{ for all }
0<\eps\leq \eps_0\leq \eps_i.
$$
Applying Proposition \ref{prop3} completes the proof of this direction of Theorem \ref{thm1}.
\end{proof}

\section{The Converse Implication}
Throughout this section we use the following notation: for a closed convex domain $K$, for all $R>0$ such that $K_R\neq\emptyset$, and for all $x\in\partial K$  let
\begin{equation}
\label{eq 51}
d_x:= |x-\pi_{K_R}(x)| = \dist(x, K_R)
\quad
\text{and}
\quad
v_x:= \frac{\pi_{K_R}(x) - x}{d_x}.
\end{equation}
The proof of Proposition \ref{prop 9} shows that $K_R$ is closed and convex. Note that if $K$ satisfies $\AIB(R,\eps)$, then for every $x\in \partial K$ we have $\dist(x,K(R)) = \dist(x,K_R) - R$ showing
\begin{equation}
\label{eq10}
R\leq d_x \leq R+\eps.
\end{equation}
To show that $K$ has boundary of class $C^{1,\alpha}$ we will show that at each point $x\in\partial K$, there is a unique inward supporting normal vector $\nu(x)$, that this normal vector is $\alpha$-H\"{o}lder continuous on all of $\partial K$, and then apply Proposition \ref{prop 5}. To this end, the next lemma allows us to  identify the relationship between $v_x$ and any inner supporting unit normal vector $\nu$ at $x\in\partial K$.
\begin{lemma}
\label{lem6}
Let $K\subset\bbbr^n$ be a closed convex domain satisfying $\AIB(R,\eps)$ for $R, \eps >0$. If $\nu$ is any inner supporting unit normal to $K$ at $x\in\partial K$, then
\begin{equation}
\label{eq11}
|\nu - v_x|^2 \leq \frac{2\eps}{R}.
\end{equation}
\end{lemma}

\begin{proof}
Let $x\in\partial K$ and let $\nu\in\bbbs^{n-1}$ satisfy \eqref{eq13}. Then, as $\pi_{K_{R}}(x) - R\nu\in \overbar B(\pi_{K_{R}}(x), R)\subset K$, we have
$$
\langle \pi_{K_{R}}(x) - R\nu - x, \nu \rangle \geq 0 \implies \langle \pi_{K_{R}}(x) - x, \nu\rangle \geq R.
$$
This in turn shows $\langle v_x, \nu\rangle \geq R/d_x$. It then follows that
$$
|\nu - v_x|^2 = 2(1-\langle \nu, v_x\rangle) \leq 2\left(1-\frac{R}{d_x}\right) = 2\left(\frac{d_x - R}{d_x}\right) \leq \frac{2\eps}{R}
$$
where the last inequality follows from \eqref{eq10}.
\end{proof}

\begin{proof}[Proof of Theorem \ref{thm1}, (ii)$\implies$ (i)]
Let $K$ be a closed convex domain satisfying (ii) and let $R = C\eps^\beta$. Fix $x\in\partial K$. If $\nu_1$ and $\nu_2$ are two distinct inner supporting unit normals of $K$ at $x$, then Lemma \ref{lem6} implies
\begin{equation}
\label{eq 52}
|\nu_1 - \nu_2| \leq 2\sqrt{\frac{2\eps}{R}} = 2\sqrt{\frac{2}{C}}\eps^{\alpha/(1+\alpha)}.
\end{equation}
Letting $\eps\to 0^+$ proves that $\nu_1 = \nu_2$, showing the uniqueness of the supporting inner unit normal of $K$ at $x$ and hence $K$ has boundary of class $C^1$ by Proposition \ref{prop 2}. Now let $x,y\in\partial K$ and if $x=y$, the result is obvious. Assume $0<|x-y|^{1+\alpha}\leq \eps_0$, and let $\eps = |x-y|^{1+\alpha}$, so that $\eps\leq \eps_0$. Applying both \eqref{eq16} and \eqref{eq11} we have
$$
|\nu(x) - \nu(y)|\leq |\nu(x) - v_x| + |v_x - v_y| + |v_y - \nu(y)| \leq 2\sqrt{\frac{2\eps}{R}} + \frac{|x-y|}{R}
$$
and substituting the values of $\eps = |x-y|^{1+\alpha}$ and $R = C\eps^\beta$ yields
\begin{equation}
\label{eq 54}
|\nu(x) - \nu(y)|\leq \left(2\sqrt{\frac{2}{C}} + \frac{1}{C}\right)|x-y|^\alpha.
\end{equation}
If $|x-y|^{1+\alpha}>\eps_0$, then we have
$$
|\nu(x) - \nu(y)|\leq 2 \leq \frac{2}{\eps_0^\frac{\alpha}{1+\alpha}}|x-y|^\alpha.
$$
This proves $\nu$ is $\alpha$-H\"{o}lder continuous and applying Proposition \ref{prop 5} concludes the proof.
\end{proof}


\section{Proof of Corollary \ref{cor1}}
The proof of Corollary \ref{cor1}(a) is an easy application of Lemma \ref{lem 17}. To prove Corollary \ref{cor1}(b), we need to apply useful parts of the proofs of Proposition \ref{prop 5}, Lemma \ref{lem6} and Theorem \ref{thm1} (ii)$\implies$(i).

\begin{proof}[Proof of Corollary \ref{cor1}]
The proof of (a) is simple, as for any $x\in\bbbr^n$ and $r>0$ we can find an $R>0$ such that $\overbar B(x,r)\subset B(0,R)$ and applying Lemma \ref{lem 17} to $f$ restricted to $B(0,2R)$ we have that $\epi(f)$ satisfies $\AIB_{B(0,R)}(C\eps^\beta,\eps)$ for all $0<\eps\leq\eps_0$, where no convexity is needed, and hence $\epi(f)$ satisfies $\AIB_{\overbar B(x,r)}(C\eps^\beta,\eps)$ for all $0<\eps\leq\eps_0$. Therefore the main part of the proof is proving (b). 

Fix $a\in\bbbr^n$ and $r>0$. By the hypothesis, there exist $C, \eps_0>0$ such that $\epi(f)$ satisfies $\AIB_{\overbar B(a,2r)}(C\eps^\beta, \eps)$ for all $0<\eps\leq \eps_0$. Let $R=C\eps^\beta$ and note that by applying the arguments establishing \eqref{eq10} and \eqref{eq11}, using the closed convex domain $K = \epi(f)$, we have for all $x\in B(a,2r)$,
$$
R\leq d_{\Gamma_f(x)} \leq R+\eps,
\quad
|\nu - v_{\Gamma_f(x)}|^2 \leq \frac{2\eps}{R}
$$
where $\nu$ is any inward supporting unit normal at $\Gamma_f(x)$ and $d_{\Gamma_f(x)}$ and $v_{\Gamma_f(x)}$ are defined in \eqref{eq 51}. Then by \eqref{eq 52} and the proof of Theorem \ref{thm1} (ii)$\implies$(i), there is a unique inward supporting unit normal at $\Gamma_f(x)$, denoted $\nu(\Gamma_f(x))$, implying a unique supporting hyperplane of $f$ at $\Gamma_f(x)$. Hence $f\in C^1(B(a,2r))$. Letting $A = \max_{x\in\overbar B(a,r)} |\nabla f(x)|$, by \eqref{eq 53} we know
\begin{equation}
\label{eq 55}
|\Gamma_f(x) - \Gamma_f(y)|\leq \sqrt{1+A^2}|x-y|
\end{equation}
for all $x,y\in \overbar B(a,r)$. Applying the argument from Theorem \ref{thm1} (ii)$\implies$(i) that produced \eqref{eq 54} and combining this with \eqref{eq 55} implies the existence of some $L>0$ satisfying
$$
|\nu(\Gamma_f(x)) - \nu(\Gamma_f(y))| \leq L|\Gamma_f(x) - \Gamma_f(y)|^\alpha \leq L\sqrt{1+A^2}|x-y|^\alpha.
$$
As in Proposition \ref{prop 5}, by \eqref{eq 63}, the function $\phi(z) = -z'/z_{n+1}$ is Lipschitz on 
$$
\{\nu(\Gamma_f(x)): x\in B(a,r)\} \subset \left\{ z\in \bbbr^{n+1}: |z|\leq 1, \, z_{n+1} \geq (1+A^2)^{-1/2}\right\}
$$
and $\nabla f (x) = \phi(\nu(\Gamma_f(x)))$ is the composition of a Lipschitz function $\phi$ and an $\alpha$-H\"{o}lder continuous function $x\mapsto \nu(\Gamma_f(x))$ on $B(a,r)$, showing $\nabla f$ is $\alpha$-H\"{o}lder continuous on $B(a,r)$. Therefore $f\in C^{1,\alpha}(B(a,r))$ proving that $f\in C^{1,\alpha}_{\rm loc}(\bbbr^n)$.

\end{proof}


\section{Regularity of the Convex Envelope}

For a function $f:\bbbr^n\to\bbbr$ bounded below by an affine function, define the convex envelope of $f$ as
$$
\conv(f)(x):= \sup\{g(x): g\leq f \text{ and } g \text{ is convex } \}.
$$

The next result is standard in convex analysis, as an application of the Carath\'{e}odory Theorem to the set $\co(\epi(f))$ and a reduction of the number of points in an epigraph representation. Proofs can be found in \cite{dacorogna, rockafellar}.
\begin{theorem}
\label{thm4}
If $f:\bbbr^n\to\bbbr$ is bounded below by an affine function, then
$$
\conv(f)(x) = \inf\left\{\sum_{i=1}^{n+1} \lambda_i f(x_i): \lambda_i \geq 0, \, \sum_{i=1}^{n+1}\lambda_i = 1, \, \sum_{i=1}^{n+1}\lambda_i x_i = x\right\}.
$$
\end{theorem}

We say a function $f:\bbbr^n\to\bbbr$ is coercive if $f(x)\to\infty$ as $|x|\to\infty$. The next lemma is the geometric core to the proof of Theorem \ref{thm2} and demonstrates the usefulness of using closed balls in the epigraph to establish regularity of a function. The proof of Lemma \ref{lem8} begins by making use of Theorem \ref{thm4} in a similar fashion to \cite[Theorem 5.4]{kiselman} and \cite[Theorem 2.3]{fathi}, that is, representing the convex envelope through use of sequences of convex combinations of the original function. The novelty of the proof lies in noting that if $f$ satisfies an approximate inner ball condition, applying the approximate inner balls to this sequence shows $\conv(f)$ inherits a comparable approximate inner ball condition, so that $\conv(f)$ retains the regularity of $f$.
\begin{lemma}
\label{lem8}
Let $f:\bbbr^n\to\bbbr$ be continuous and coercive. For every nonempty compact $V\subset\bbbr^n$, there is $M>0$ with the following property: if $\epi(f)$ satisfies $\AIB_{\overbar B(0,M)}(R,\eps)$, then $\epi(\conv(f))$ satisfies $\AIB_{V}(R/(n+1),\eps)$.
\end{lemma}

\begin{proof}
By continuity and coercivity, $f$ is bounded below. Hence we can add a constant to $f$ so that $f\geq 0$. Let $A:=\max_{V}f$. Then by coercivity, we can choose $M>0$ large enough so that $V\subset B(0,M)$ and
\begin{equation}
\label{eq15}
f(z) > (n+1)(A+1)
\quad
\text{ if }
|z|\geq M.
\end{equation}

Fix $y\in V$. By Theorem \ref{thm4}, we can choose $\lambda_i^k\geq 0$ and $y_i^k\in\bbbr^n$ such that
$$
\sum_{i=1}^{n+1} \lambda_i^k = 1,
\quad
y=\sum_{i=1}^{n+1}\lambda_i^k y_i^k,
\quad
\sum_{i=1}^{n+1}\lambda_i^k f(y_i^k) \to \conv(f)(y) \leq A
$$
with $\lambda_1^k \geq \cdots \geq \lambda_{n+1}^k$ for all $k\in\bbbn$. Thus, $\lambda_1^k \geq 1/(n+1)$. Moreover there exists $K_0\in\bbbn$ such that $|y_1^k|<M$ for all $k\geq K_0$. Indeed, for large $k$, recalling $\lambda_i^k, f(y_i^k)\geq 0$, we have 
$$
\lambda_1^k f(y_1^k)\leq \sum_{i=1}^{n+1}\lambda_i^k f(y_i^k) \leq A+1 
$$
and $\lambda_1^k \geq 1/(n+1)$ shows that $f(y_1^k)\leq (n+1)(A+1)$. Then applying \eqref{eq15} shows $|y_1^k|<M$. We then apply the assumption that $\epi(f)$ satisfies $\AIB_{\overbar B(0,M)}(R,\eps)$ to the specific points $(y_1^k, f(y_1^k))$, that is, for each $k\geq K_0$ there exists $h_1^k\in\epi(f)$ such that 
\begin{equation}
\label{eq17}
\overbar B(h_1^k, R) \subset \epi(f)
\quad
\text{ and }
\quad
\dist(\Gamma_f(y_1^k), \overbar B(h_1^k, R) ) \leq \eps.
\end{equation}
Define the set
$$
B^k := \lambda_1^k \overbar B(h_1^k, R) + \sum_{i=2}^{n+1} \lambda_i^k \{\Gamma_f(y_i^k)\}
$$
and note that $B^k$ is a closed ball of radius $\lambda_1^k R$. Moreover, $B^k\subset\epi(\conv(f))$ because $\epi(\conv(f))$ is convex and contains $\epi(f)$. We note that the centers of $B^k$, given by $h^k:=\lambda_1^k h_1^k + \sum_{i=2}^{n+1} \lambda_i^k \Gamma_f(y_i^k)$, are bounded as they differ by at most $R+\eps$ from the convergent sequence $\sum_{i=1}^{n+1}\lambda_i^k \Gamma_f(y_i^k)$. Given that both $\lambda_1^k$ and $h^k$ are bounded, by passing to a subsequence, we have that $h^k\to h$ and $\lambda_1^k \to \lambda \geq 1/(n+1)$. As $\conv(f)\geq 0$ is convex and continuous, $\epi(\conv(f))$ is closed which implies
$$
\overbar B(h, \lambda R)\subset \epi(\conv(f)).
$$
By \eqref{eq17}, we can choose $a_1^k\in \overbar B(h_1^k, R)$ such that $|a_1^k - \Gamma_f(y_1^k)|\leq \eps$, and then define the point $a^k:=\lambda_1^k a_1^k + \sum_{i=2}^{n+1} \lambda_i^k \Gamma_f(y_i^k)$. We note that $a^k\in B^k$ for each $k$. As in the argument for the boundedness of $h^k$, the sequence $(a^k)$ is bounded since it differs by at most $\eps$ from the convergent sequence $\sum_{i=1}^{n+1}\lambda_i^k \Gamma_f(y_i^k)$. Thus, passing to another subsequence, we have that $a^k\to a\in \overbar B(h,\lambda R)$ where $|a - \Gamma_{\conv(f)}(y)| \leq \lambda \eps \leq \eps$. This shows that for every $y\in V$, there exists $\overbar B(h, \lambda R)\subset \epi(\conv(f))$ such that $\dist(\Gamma_{\conv(f)}(y), \overbar B(h,\lambda R)) \leq \eps$.  The result follows by noting that for each $y\in V$ we can then find a closed ball $B_y$ of radius $R/(n+1)\leq\lambda R$, such that $B_y\subset \overbar B(h, \lambda R)$, though not necessarily concentric with $\overbar B(h,\lambda R)$, and $\dist(\Gamma_{\conv(f)}(y), B_y) \leq \eps$, so that $\epi(\conv(f))$ satisfies $\AIB_V(R/(n+1), \eps)$.
\end{proof}

\begin{proof}[Proof of Theorem \ref{thm2}]
Fix $x\in\bbbr^n$ and $r>0$ and choose $M$ as in Lemma \ref{lem8} with $V = \overbar B(x,r)$. Corollary \ref{cor1}(a) applied to the closed ball $\overbar B(0,M)$ provides the existence of $C,\eps_0>0$ such that $\epi(f)$ satisfies $\AIB_{\overbar B(0,M)}(C\eps^\beta, \eps)$ for all $0<\eps\leq \eps_0$. Lemma \ref{lem8} then shows that $\epi(\conv(f))$ satisfies $\AIB_V(C\eps^\beta/(n+1), \eps)$ and Corollary \ref{cor1}(b) completes the proof.
 \end{proof}

 \section{Regularity of the Sum of Convex Domains}
The approximate inner ball condition is inherited by the Minkowski sum, and hence the proof of Theorem \ref{thm3} follows almost immediately from Theorem \ref{thm1}.
\begin{proof}[Proof of Theorem \ref{thm3}]
Let $\beta:= (1-\alpha)/(1+\alpha)$. As $A$ is a convex body and $B$ is convex and closed, $A+B$ is a closed convex domain. By Theorem \ref{thm1}, there are constants $C,\eps_0>0$ such that every point $a\in A$ lies within an $\eps$ distance of a closed ball $\overbar B(h,C\eps^\beta)\subset A$ for all $0<\eps\leq \eps_0$. Take any $z= a+b\in A+B$. Translating $\overbar B(h,C\eps^\beta)$ by the point $b$ yields $\overbar B(h+b, C\eps^\beta)\subset A+B$, and
$$
\dist(z,\overbar B(h+b, C\eps^\beta)) = \dist(a, \overbar B(h, C\eps^\beta))\leq \eps.
$$
Therefore $A+B$ satisfies $\AIB(C\eps^\beta,\eps)$ for every $0<\eps\leq \eps_0$ and applying Theorem \ref{thm1} concludes the proof.
\end{proof}

\section*{Acknowledgments}
The author thanks Piotr Haj{\l}asz for his guidance with developing the ideas in this paper and for being a great mentor.

\section*{Declaration of generative AI use}
The results in this paper were developed as part of the doctoral thesis of the author, from 2021 - 2024, before the use of LLM's were as widespread. During the preparation of this manuscript, the author used OpenAI's ChatGPT to assist with literature searches, exposition,
formatting, and the review of mathematical arguments. The author takes full responsibility for the mathematical results, proofs, references, and final content of the manuscript.

\end{document}